\documentclass[12pt]{amsart}
\usepackage{tikz-cd}
\usepackage{fancyhdr} 
   
\usepackage{latexsym, amssymb, amscd, amsfonts,pstricks,enumitem,amsmath, young}
\usepackage{latexsym, amssymb, amscd, amsfonts,pstricks,enumitem,amsmath}
\usepackage{amsmath,amsthm,amssymb,mathrsfs}
 
\usepackage[all]{xypic}
\usepackage{ifthen}
\usepackage{longtable}
 
\renewcommand{\subsection}[1]{\vspace{3mm}\refstepcounter{subsection}\noindent{\bf \thesubsection. #1.} }

\newcommand{\np}{\vspace{3mm}\refstepcounter{subsection}\noindent{\bf \thesubsection. }}

\renewcommand{\subsubsection}[1]{\vspace{3mm}\refstepcounter{subsubsection}\noindent{\bf \thesubsubsection. #1.} }

\numberwithin{equation}{section}

\newcommand{\Hom}{\operatorname{Hom}}
\renewcommand{\geq}{\geqslant}
\renewcommand{\leq}{\leqslant}
\newcommand{\Osh}{{\mathcal O}}                        

\renewcommand{\H}{\mathrm{H}}                          
                            
\newcommand{\K}{\mathrm{K}}                
\newcommand{\M}{\mathrm{M}}
\newcommand{\N}{\mathrm{N}}            

\newcommand{\kk}{\mathbf{k}}

\newcommand{\Vol}{\operatorname{Vol}}

\newcommand{\PP}{\mathbb{P}} 
\newcommand{\RR}{\mathbb{R}} 
\newcommand{\ZZ}{\mathbb{Z}} 

\newtheorem{theorem}{Theorem}[section]

\newtheorem{proposition}[theorem]{Proposition}

\theoremstyle{definition}

\begin{document}

\title[Generalized torus invariant GCD]{
Generalized GCD for toric Fano varieties}
\author{Nathan Grieve}
\address{Department of Mathematics, Michigan State University,
East Lansing, MI, USA, 48824}
\email{grievena@msu.edu}%

\begin{abstract} 
We study the greatest common divisor problem for torus invariant blowing-up morphisms of nonsingular toric Fano varieties.  Our main result applies the theory of Okounkov bodies together with an arithmetic form of Cartan's Second Main theorem, which has been established by Ru and Vojta.  It also builds on Silverman's geometric concept of greatest common divisor.  As a special case of our results, we deduce a bound for the generalized greatest common divisor of pairs of nonzero algebraic numbers.
\end{abstract}

\maketitle

\section{Introduction}

\np  Let $\kk$ be a number field, $X$ a nonsingular projective variety over $\kk$ and $Y \subsetneq X$ a nonsingular codimension $r$ subvariety, $r \geq 2$.  As observed by Silverman \cite{Silverman:2005}, the height function $h_E(\cdot)$ for $E$ the exceptional divisor of
$$
\pi \colon X' = \operatorname{Bl}_Y(X) \rightarrow X
$$
the blowing-up of $X$ along $Y$ can be interpreted as a generalized logarithmic greatest common divisor.  In particular, it is of interest to obtain bounds for $h_E(\cdot)$.

\np  The purpose of the present note, is to obtain results of this flavour for torus invariant blowing-up morphisms.  These results are consequences of the Main Arithmetic General Theorem, \cite{Ru:Vojta:2016}, together with the theory of polytopes for torus invariant divisors.  

\np  Our main result is stated as Theorem \ref{unconditional:arithmetic:general:theorem:gcd:bound}, which should also be viewed within the context of our more computational algorithmic point of view which we undertake in Sections \ref{Toric:Calculations} and \ref{Torus:expected:order:vanishing}.  Within the context of toric varieties, Theorem \ref{unconditional:arithmetic:general:theorem:gcd:bound} becomes rather explicit and algorithmic.  Indeed, developing that point of view is the content of Section \ref{Torus:expected:order:vanishing}.

\np As one special case of what we do here, we obtain a result which gives an unconditional upper bound for the greatest common divisor of pairs of nonzero algebraic numbers.  We state that result as Theorem \ref{gcd:bound:intro:thm'}.  It is a consequence of a more general result (see Theorem \ref{unconditional:arithmetic:general:theorem:gcd:bound}) together with Proposition \ref{beta:gcd:prop}.  

\np  The proof of Proposition \ref{beta:gcd:prop} involves calculations with Okounkov bodies; see Sections \ref{Toric:Calculations} and \ref{Torus:expected:order:vanishing}.  We believe that these methods are also of an independent interest.   Indeed, they may be seen, for example, as special cases of our related more general results that we obtained in \cite{Grieve:chow:approx}.  They also have an interpretation in terms of certain invariants that arise within the context of K-stability.  For a more detailed discussion about those topics, we refer to the works \cite{Grieve:2018:autissier}, \cite{Donaldson:2002}, \cite{Boucksom-Hisamoto-Jonsson:2016} and the references therein. 

\section{Arithmetic Preliminaries}

\np  We fix arithmetic notation and conventions which resemble those of \cite{Bombieri:Gubler}.  Let $\kk$ be a number field and $M_\kk$ its set of places.  We choose absolute values $|\cdot|_v$, for each $v \in M_\kk$, in such a way that the product formula holds true with multiplicities equal to one.  We also let $S$ be a fixed finite set of places.%

\np 
If $D$ is a Cartier divisor on an (irreducible) projective variety $X$ over $\kk$ then we denote the local Weil function of $D$ with respect to a place $v \in M_\kk$ by 
$$\lambda_{D}(\cdot,v) = \lambda_{D,v}(\cdot).$$  
We also denote the logarithmic height function of such a divisor by 
$$h_{D}(\cdot) = h_{\Osh_X(D)}(\cdot).$$
Again, we refer to \cite{Bombieri:Gubler} for more details about Weil and height functions.

\np  In our present notation, the greatest common divisor of a pair of nonzero integers 
$0 \not =  \alpha,\beta \in \ZZ$
is governed by the formula
$$
\log \operatorname{gcd}( \alpha,\beta )  := \sum_{\text{prime numbers } p} \min \{ \operatorname{ord}_p( \alpha), \operatorname{ord}_p(\beta) \} \log p 
$$
\cite{Silverman:2005}.  More generally, it makes sense to define the \emph{generalized logarithmic greatest common divisor} for pairs of nonzero algebraic numbers
$\alpha, \beta \in \kk^\times$. 
To this end, as in \cite[page 337]{Silverman:2005}, we put
\begin{equation}\label{log:gcd:defn}
h_{\operatorname{gcd}}(\alpha,\beta) := \sum_{v \in M_\kk} \min \{ \max\{ -\log |\alpha|_v,0\}, \max \{  -\log |\beta|_v, 0\} \} \text{.}
\end{equation}
In particular, for pairs of nonzero integers $0 \not = \alpha,\beta \in \ZZ$,  
 it follows from \eqref{log:gcd:defn} that
$$
h_{\operatorname{gcd}}(\alpha,\beta) = \log \operatorname{gcd}(\alpha,\beta) \text{.}
$$

\np  There is an important \emph{geometric interpretation} of the generalized greatest common divisor  \eqref{log:gcd:defn}.  This point of view is emphasized by Silverman, for example \cite{Silverman:2005} and \cite[page 204]{Silverman:1987} (see also \cite[Proposition 3]{Yasufuku:2011} and \cite[Proposition 5]{Yasufuku:2012}).  For example, let
$$
S(\Sigma) := \PP^1_{\kk} \times \PP^1_{\kk}
$$
and 
$$
\pi \colon \operatorname{Bl}_{\{\mathrm{pt \}}}(S(\Sigma)) \rightarrow S(\Sigma)
$$
the blowing-up of $S(\Sigma)$ at the origin 
$\{\mathrm{pt} \} \in S(\Sigma)$
with exceptional divisor $E$.  

Fixing local height functions $\lambda_{E,v}(\cdot)$, for each $v \in M_\kk$, it was noted in \cite{Silverman:2005} that if 
$$
(\alpha,\beta) := [1:\alpha] \times [1 : \beta] \in S(\Sigma) \text{, }
$$
with $\alpha,\beta \in \kk^\times$, then
\begin{align*}
h_{\operatorname{gcd}}(\alpha,\beta)  = \sum_{v \in M_\kk} \lambda_{E,v}(\pi^{-1}(\alpha,\beta)) + \mathrm{O}(1)  = h_E(\pi^{-1}((\alpha,\beta))) \text{.}
\end{align*}

\np  Our arithmetic results are consequences of the Main Arithmetic General Theorem \cite{Ru:Vojta:2016}.  

\begin{theorem}[{\cite{Ru:Vojta:2016}}]\label{main:arithmetic:general:thm}
Let $X$ be a projective variety over a number field $\kk$.  Let $D_1,\dots,D_q$ be nonzero effective Cartier divisors on $X$, defined over $\kk$.  Assume that the divisors $D_1,\dots,D_q$ intersect properly and put $D = D_1 + \dots + D_q$.  Let $L$ be a big line bundle on $X$ and defined over $\kk$.  Let $S \subseteq M_\kk$ be a finite set of places.  Then, for each $\epsilon > 0$, the inequality
$$
\sum_{v \in S} \lambda_D(x,v) \leq \left( \max_{1 \leq j \leq q} \left\{ \frac{1}{\beta(L,D_j)} \right\} + \epsilon \right) h_{L}(x) + \mathrm{O}(1) 
$$
holds true for all $\kk$-rational points $x \in X(\kk)$ outside of some proper Zariski closed subset of $X$.
\end{theorem}

\np\label{general:thm:assumptions}
Here, in Theorem \ref{main:arithmetic:general:thm}, 
\begin{equation}\label{beta:eqn}
\beta(L,D_j) := \int_0^\infty \frac{\Vol(L - t D_j ) }{ \Vol(L) } \mathrm{d}t \text{, }
\end{equation}
for $j = 1,\dots,q$, denotes the expected order of vanishing of $L$ along $D_j$.  This form of Theorem \ref{main:arithmetic:general:thm}, stated using the asymptotic volume constants \eqref{beta:eqn}, has also been noted, for example, in  \cite[Theorem 1.8]{Ru:Wang:2016}  and \cite[Theorem 1.1]{Grieve:2018:autissier}.  Finally, that the divisors $D_1,\dots,D_q$ intersect properly, means that for each subset 
$I \subseteq \{1,\dots,q\}$ 
and all 
$x \in \bigcap_{i \in I} \operatorname{Supp} D_i$, 
the sequence 
$(f_i)_{i \in I}$,
for $f_i$ a local defining equation for $D_i$, is a regular sequence in $\Osh_{X,x}$, the local ring of $X$ at $x$.

\section{The arithmetic general theorem for Fano varieties and consequences for greatest common divisors}\label{arithmetic:general:thm:gcd}

\np
In this section, we discuss the geometric aspect to the generalized greatest common divisor problem.  What we discuss here is mostly based on  \cite{Silverman:2005}.  In particular, our main aim is to derive the inequality \eqref{gcd:eqn12} below.  This is an unconditional gcd bound, with more restrictive assumptions.  It may be seen by analogy with that which was observed in \cite[Theorem 6]{Silverman:2005}.  In what follows, when no confusion is likely, we identify Cartier divisors with linear equivalence classes thereof.  On the other hand, note that in applications of the Main Arithmetic General Theorem (Theorem \ref{main:arithmetic:general:thm}), for example, it is important to make the distinction between a given Cartier divisor and its linear equivalence class.

\np  Let $X$ be a nonsingular (irreducible) projective variety, over a number field $\kk$, and let 
$Y \subsetneq X$ 
be a nonsingular codimension $r$ subvariety, $r \geq 2$.  Let 
\begin{equation}\label{gcd:eqn1}
\pi \colon X' = \operatorname{Bl}_Y(X) \rightarrow X
\end{equation}
be the blowing-up of $X$ along $Y$ with exceptional divisor $E$.  Then $X'$ is nonsingular and, furthermore, the respective canonical divisor classes of $X$ and $X'$ are related by the formula
\begin{equation}\label{gcd:eqn2}
\K_{X'} = \pi^* \K_X + (r-1)E.
\end{equation}

\np\label{transversality:Vojta:assumption}  Hencefourth, we impose the following two assumptions
\begin{enumerate}
\item[(i)]{there exist nonzero effective Cartier divisors $D_1,\dots,D_q$ on $X$ for which
$$
-\K_X \sim_{\text{lin. equiv.}} D_1+\dots+D_q 
$$
and which have the property that the divisors $\pi^*D_1,\dots,\pi^*D_q$ intersect properly on $X'$; and
}
\item[(ii)]{
the divisor $-\K_{X}$ is big.
}
\end{enumerate}

\np In particular, we may assume that
$$
-\pi^*\K_X = \pi^*D_1 + \dots + \pi^*D_q \text{;}
$$
set
\begin{equation}\label{gcd:eqn4}
\gamma := \max_{1 \leq j \leq q} \left\{ \frac{1}{\beta(-\K_{X'},\pi^*D_j)} \right\} \text{.}
\end{equation}
Fix a real number $\delta \geq 0$, which depends on the given choice of divisors $D_1,\dots,D_q$, and which satisfies the condition that
\begin{equation}\label{gcd:eqn5}
\gamma \leq 1 + \delta.
\end{equation}

\np  Let $\epsilon > 0$.  By the Main Arithmetic General Theorem \cite{Ru:Vojta:2016}, which applies because of our assumptions (i) and (ii), in particular, because the divisors $\pi^*D_1,\dots, \pi^*D_q$ intersect properly and because $-\K_X$ is big, we have that
\begin{equation}\label{gcd:eqn6}
- \sum_{v \in S} \lambda_{\pi^*\K_X}(v,x') \leq (\gamma + \epsilon) h_{-\K_{X'}}(x') + \mathrm{O}(1) \leq (1 + \delta + \epsilon) h_{-\K_{X'}}(x') + \mathrm{O}(1)
\end{equation}
for all $\kk$-rational points
$
x' \in X'(\kk) \setminus Z'(\kk)
$
and
$
Z' \subsetneq X'
$
some Zariski closed proper subset.  Using the relation \eqref{gcd:eqn2}, in the equivalent form
$$
- \K_{X'} = - \pi^* \K_X - (r-1)E\text{,}
$$
we may rewrite the inequality \eqref{gcd:eqn6} as
\begin{equation}\label{gcd:eqn8}
- \sum_{v \in S} \lambda_{\pi^*\K_X}(v,x') + h_{\K_{X'}}(x') \leq (\delta + \epsilon)h_{-\pi^* \K_X}(x') - (\delta + \epsilon)(r-1)h_E(x') + \mathrm{O}(1) \text{.}
\end{equation}
By rearranging \eqref{gcd:eqn8}, we then obtain that
\begin{equation}\label{gcd:eqn9}
- \sum_{v \in S} \lambda_{\pi^*\K_X}(v,x') + h_{\pi^*\K_X}(x') + (1+\delta+\epsilon)(r-1)h_E(x') \leq (\delta + \epsilon) h_{-\pi^* \K_X}(x') + \mathrm{O}(1) \text{.}
\end{equation}
Recall, that the height function $h_{\pi^*\K_X}(\cdot)$ may be expressed in terms of local Weil functions.  Precisely
\begin{equation}\label{gcd:eqn10}
h_{\pi^*\K_X}(\cdot) = \sum_{v \in S} \lambda_{\pi^* \K_X}(v,\cdot) + \sum_{v \in M_{\kk} \setminus S} \lambda_{\pi^* \K_X}(v,\cdot) + \mathrm{O}(1).
\end{equation}
Using this relation \eqref{gcd:eqn10}, the inequality \eqref{gcd:eqn9} takes the form
\begin{equation}\label{gcd:eqn11}
\sum_{v \in M_\kk \setminus S} \lambda_{\pi^* \K_X}(v,x') + (1 + \delta + \epsilon)(r-1) h_E(x') \leq (\delta + \epsilon) h_{-\pi^*\K_X}(x') + \mathrm{O}(1).
\end{equation}
But now, we may use \eqref{gcd:eqn11} to isolate for $h_E(x')$.  In doing so, we obtain the inequality
\begin{equation}\label{gcd:eqn12}
h_E(x') \leq \frac{1}{(1 + \delta + \epsilon)(r-1)} \left( (\delta + \epsilon) h_{-\pi^* \K_X}(x') - \sum_{v \in M_\kk \setminus S} \lambda_{\pi^* \K_X}(v,x')\right) + \mathrm{O}(1) \text{.}
\end{equation}

\np  For later use, we summarize the above discussion in the following way.  

\begin{theorem}\label{unconditional:arithmetic:general:theorem:gcd:bound}
Let $\pi \colon X' \rightarrow X$
be the blowing-up morphism of a nonsingular projective variety $X$ along a nonsingular codimension $r$ subvariety 
$
Y \subsetneq X\text{,}
$
 $r \geq 2$.  Let $E$ be the exceptional divisor of $\pi$.  Assume that: 
 \begin{enumerate}
 \item[(i)]{
 there exist nonzero effective Cartier divisors $D_1,\dots,D_q$ on $X$ which have the two properties that:
 \begin{enumerate}
 \item[(1)]{
 the divisors $\pi^*D_1,\dots,\pi^*D_q$ intersect properly on $X'$; and
 }
 \item[(2)]{
 the anticanonical class of $X$ is represented by $D_1 + \dots + D_q$; and
 }
 \end{enumerate}
 }
 \item[(ii)]{
 the anticanonical class $-\K_X$ is big.
 }
 \end{enumerate}
 In this context, set
 $$
 \gamma := \max_{1 \leq j \leq q} \left\{ \frac{1}{\beta(-\K_{X'}, \pi^* D_j) } \right\}
 $$
 and fix a real number $\delta \geq 0$, as in \eqref{gcd:eqn5}, which depends on the given choice of divisors $D_1,\dots,D_q$, and which satisfies the condition that
 $$
 \gamma \leq 1 + \delta \text{.}
 $$
Let $\epsilon > 0$.  Then, with these assumptions, there exists a
proper Zariski closed subset $Z' \subsetneq X'$ which has the property that the inequality
\begin{equation}\label{gcd:eqn12'}
h_E(x') \leq \frac{1}{(1 + \delta + \epsilon)(r-1)} \left( (\delta + \epsilon) h_{-\pi^* \K_X}(x') - \sum_{v \in M_\kk \setminus S} \lambda_{\pi^* \K_X}(v,x')\right) + \mathrm{O}(1) 
\end{equation}
is valid for all $\kk$-rational points
$
x' \in X'(\kk) \setminus Z'(\kk).
$
\end{theorem}

\begin{proof}
By assumption, the two conditions (i) and (ii) stated in \S \ref{transversality:Vojta:assumption} are satisfied.   Thus, the inequality \eqref{gcd:eqn12'} is implied by the discussion that precedes and includes the inequality \eqref{gcd:eqn12}.
\end{proof}

\np {\bf Remark.}  In \cite{Silverman:2005}, the logarithmic height function $h_{E}(\cdot)$ is called the \emph{generalized greatest common divisor height}.  The inequality \eqref{gcd:eqn12'} is similar to the main strong upper bound for the generalized greatest common divisor heights that are made possible by Vojta's Main Conjecture.  The above two conditions given in \S \ref{transversality:Vojta:assumption}, are closely related to the concept of \emph{admissible pairs} and a logarithmic formulation of Vojta's Main Conjecture \cite[Section 1.2]{Levin:GCD}.  They also should be compared with the discussion of \cite[Section 4]{Silverman:2005}.

\section{Preliminaries about toric varieties}\label{Toric:Calculations}

\np  A main aim of the present article is to obtain an explicit form of Theorem \ref{unconditional:arithmetic:general:theorem:gcd:bound} for the case of nonsingular Fano toric varieties.  In this section, we fix notations and conventions and recall relevant facts, about toric varieties, which are especially important for our purposes here.  For more details, we refer to \cite{Ful:toric} and \cite{CLS}. 

\np
For example, such nonsingular toric varieties
$
X = X(\Sigma) \text{, }
$
defined over a fixed base field $\kk$, are determined by a fan
$
\Sigma \subseteq \N_{\RR}
$
for 
$
\N \simeq \mathbb{Z}^d \text{, }
$
where
$
d := \dim X \text{.}
$
Recall that the condition that $X$ is smooth is equivalent to the condition that the minimal generators of each (strongly convex rational polyhedral) cone 
$\sigma \in \Sigma$ 
form part of a $\mathbb{Z}$-basis for $\N$.
Henceforth, we also put
$
\M := \Hom_{\ZZ}(\N,\ZZ)
$
and
$
\M_{\RR} := \Hom_{\RR}(\N_\RR,\RR) \text{.}
$

\np
Let $v_1,\dots,v_r$ be the primitive generators for the one dimensional cones in $\Sigma$.  Recall, that each $v_i$ corresponds to a prime torus invariant divisor $D_i$, for $i = 1, \dots, r$.  Furthermore, the canonical divisor of $X$ is 
$$
\K_X = - \sum_{i=1}^r D_i \text{.}
$$ 
In particular, the  anticanonical divisor 
$$
-\K_X = \sum_{i=1}^r D_i
$$
is a strict normal crossings (snc) divisor.

\np  In general, fixing a divisor
$
D = \sum_{i = 1}^r a_i D_i \text{, }
$
with $a_i \in \RR$, we obtain its  polyhedron
$$
\mathrm{P}_D := \{ m \in \M_{\RR} : \langle m , v_i \rangle \geq - a_i \text{, for $i = 1, \dots, r$} \}\text{.}
$$
A key point to the theory, is that the lattice points of $\mathrm{P}_D$, that is the set
$$
\mathrm{P}_D \cap \M \text{, }
$$
give a torus invariant basis of $\H^0(X, \Osh_X(D))$.

\np  We may construct an \emph{admissible flag} $Y_\bullet$, in the sense of \cite[Section 1]{Lazarsfeld:Mustata:2009}, that consists of torus invariant subvarieties of $X$.  The idea is to fix an ordering of the prime torus invariant divisors in such a way that
$$
Y_i = D_1 \cap \dots \cap D_i \text{, }
$$
for $i \leq d$.  

\np  Indeed, in this way, we obtain a flag
$$
Y_\bullet \colon X = Y_0 \supseteq Y_1 \supseteq \hdots \supseteq Y_{d-1} \supseteq Y_d = \{ \mathrm{pt} \}\text{,}
$$
of irreducible (torus invariant) subvarieties of $X$, which has the two properties that 
\begin{enumerate}
\item[(i)]{ 
$\operatorname{codim}_X(Y_i) = i$; and
}
\item[(ii)]{ each $Y_i$ is nonsingular at the point $\{\mathrm{pt}\}$.
}
\end{enumerate}

\np  Now, following \cite[Section 6.1]{Lazarsfeld:Mustata:2009}, to relate the Okounkov body of $\Delta(D)$, for $D$ a big divisor on $X$, to its polytope $\mathrm{P}_D$, note that the ray generators $v_1,\dots, v_d$ determine a  maximal cone $\sigma$, in $\Sigma$; they also determine an isomorphism
$$
\phi_{\RR} \colon \M_{\RR} \xrightarrow{\sim} \RR^d \text{,}
$$
which is defined by
$$
u \mapsto \left( \langle u , v_i \rangle \right)_{1 \leq i \leq d} \text{.}
$$

\np  Fixing a big Cartier divisor $D$ on $X$, we may choose such a torus invariant flag $Y_\bullet$ with the property that the restriction of $D$ to $U_\sigma$, the affine open subset determined by $\sigma$ is trivial
$$
D|_{U_\sigma} = 0 \text{.}
$$
With these conventions, as noted in \cite[Proposition 6.1]{Lazarsfeld:Mustata:2009}, the Okounkov body $\Delta(D)$ of $D$ with respect to the flag $Y_\bullet$ relates to the polytope $\mathrm{P}_D$ via the rule
$$
\Delta(D) = \phi_{\RR}(\mathrm{P}_D) \text{.}
$$
Finally for later use, we recall that
$$
\operatorname{Vol}_X(D) = d! \operatorname{Vol}_{\mathbb{R}^d}(\mathrm{P}_D) \text{.}
$$

\section{Expected order of vanishing along torus invariant divisors}\label{Torus:expected:order:vanishing}

\np  Throughout this section
$
X = X(\Sigma)
$
denotes a smooth projective toric variety defined over an algebraically closed field $\kk$.  We also fix $E$ a torus invariant prime divisor over $X$ and
\begin{equation}\label{expect:eqn2}
\pi \colon X' \rightarrow X
\end{equation}
a nonsingular model of $X$ with the property that
$
E \subseteq X' \text{.}
$  

\np  
Recall, that such divisors $E$ determine torus invariant valuations on the function field $\kk(X)$.  Moreover, $E$ may be a prime divisor on $X$; in that case, the morphism \eqref{expect:eqn2}, is simply the identity morphism.  Finally, note that $E$ may be interpreted as a (prime) birational divisor, in the sense of \cite[Section 4]{Ru:Vojta:2016}, over $X$.

\np  Our main goal is to describe a combinatorial algorithm to compute the  asymptotic volume constant
\begin{equation}\label{expect:eqn4}
\beta(L,E) := \int_0^\infty \frac{ \Vol(\pi^*L -t E)}{\Vol(L) } \mathrm{d}t
\end{equation}
for
$
L = \Osh_X(D)
$
a big line bundle on $X$, determined by a big torus invariant divisor on $X$, and $E$ a torus invariant prime divisor over $X$.

\np  We may assume that the morphism \eqref{expect:eqn2} is obtained as a composition of toric blowing-up morphisms.  What is the same, the fan $\Sigma'$ of $X'$ is obtained from $\Sigma$, the fan of $X$, by way of  successive star subdivisions.

 \np\label{PP2:star:subdivision} 
{\bf  Example.}  Consider the divisor $E$ over the projective plane $\PP^2$ which is obtained by way of blowing-up a torus invariant point.  In more detail, the fan $\Sigma$ that corresponds to $\PP^2$ has primitive ray generators
$$
v_1 = (1,0), v_2 = (0,1) \text{ and } v_3 = (-1,-1) \text{.}
$$
Let 
$
\{\mathrm{pt}\} \in S = \PP^2
$
be the torus invariant point that corresponds to the maximal cone which is determined by the primitive ray vectors $v_1$ and $v_3$.  The fan $\Sigma'$ which determines
$$
\pi \colon S' = \operatorname{Bl}_{\{\mathrm{pt}\}}(S) \rightarrow S \text{, }
$$
the monoidal transformation of $S$ with centre $\{\mathrm{pt}\}$ has primitive ray generators
$$
v_0' = (0,-1) \text{, } v_1' = (1,0) \text{, } v_2' = (0,1) \text{ and } v_3' = (-1,-1) \text{. }
$$
The exceptional divisor $E$ corresponds to $v_0'$.   

Moreover, let $D_i'$ denote the prime torus invariant divisor that corresponds to $v_i'$, for $i = 1,2$ and $3$.  Then the pullbacks of the torus invariant prime divisors $D_i$, for $i = 1,2$ and $3$, which correspond to the primitive ray vectors $v_1,v_2$ and $v_3$ on $\PP^2$ are given by
$$
\text{
$
\pi^* D_1 = D_1' + E
$,
$
\pi^* D_2 = D_2'
$
and
$ 
\pi^* D_3 = D_3' + E \text{.}
$
}
$$

\np  We now describe a combinatorial algorithm to compute the asymptotic volume constant
\begin{equation}\label{expect:eqn4}
\beta(L,E) := \int_0^\infty \frac{ \Vol(\pi^*L -t E)}{\Vol(L) } \mathrm{d}t
\end{equation}
for
$
L = \Osh_X(D)
$
a big line bundle on $X$, determined by a big torus invariant divisor on $X$, and $E$ a torus invariant prime divisor over $X$.  In particular, our approach here allows for calculation of the quantities \eqref{expect:eqn4} for finite sequences of toric blowing-up morphisms over the (relatively) minimal rational surfaces $\PP^2$, $\PP^1\times \PP^1$ and $\mathbb{F}_r$, for $r \geq 2$.

\np  Let $\gamma_{\mathrm{eff}}$ denote the pseudoeffective threshold of $L$ along $E$
$$
\gamma_{\mathrm{eff}}(L,E) = \gamma_{\mathrm{eff}} := \sup \{ t : \pi^*L - t E \text{ is effective}\} \text{.}
$$
Then the quantity \eqref{expect:eqn4}  takes the form
\begin{equation}\label{expect:eqn5}
\beta(L,E) := \int_0^{\gamma_{\mathrm{eff}}} \frac{ \Vol(\pi^*L - t E)}{\Vol(L) } \mathrm{d}t \text{.}
\end{equation}

\np  Our approach to determining the quantity \eqref{expect:eqn5} may be seen as a special case of our more general results.  The idea is to study the polyhedra
$
\mathrm{P}_{\pi^*L - t E}\text{,} 
$
\text{for $0 \leq t \leq \gamma_{\mathrm{eff}}$.}

To begin with, let $v_0',\dots,v_{r}'$ be the primitive ray vectors for the fan $\Sigma'$ and suppose that $E$ corresponds to $v_0'$.  We may write $\pi^*L$ in the form 
\begin{equation}\label{expect:eqn6}
\pi^* L = \sum_{i=1}^{r} a_i D_i' + a_0 E.
\end{equation}

\np  The polyhedra
$$
\mathrm{P}_{\pi^*L - t E} \subseteq \mathrm{P}_{\pi^* L}
$$
are then cut out by the system of inequalities, for $m \in \M_\RR$,
$$
\text{
$
m \cdot v_0' \geq   -a_0 + t \text{; } 
$
$
m \cdot v_1' \geq -a_1 \text{; } 
$
\text{ \dots }; 
and
$   
m \cdot v_r' \geq - a_{r} 
$.
}
$$

\np Note that a bound for $\gamma_{\mathrm{eff}}$ may be obtained from the polyhedron $\mathrm{P}_{\pi^*L}$.  Furthermore, the expected order of vanishing of $L$ along $E$, especially the quantity \eqref{expect:eqn5}, is then obtained via the knowledge of the volumes of the $\mathrm{P}_{\pi^*L - t E}$.  Explicitly, recall that
\begin{equation}\label{expect:eqn9}
\operatorname{Vol}_{X'}(\pi^*L - t E) = d! \operatorname{Vol}_{\mathbb{R}^d}(\mathrm{P}_{\pi^*L - t E})
\end{equation}
whereas 
\begin{equation}\label{expect:eqn10}
\operatorname{Vol}_X(L) = d! \operatorname{Vol}_{\mathbb{R}^d}(\mathrm{P}_{\pi^* L})\text{.}
\end{equation}
Combining \eqref{expect:eqn9} and \eqref{expect:eqn10}, we may rewrite the expected order of vanishing, in other words the quantity \eqref{expect:eqn5}, in the form
$$
\beta(L,E) = \int_0^{\gamma_{\mathrm{eff}}} \frac{ \operatorname{Vol}_{\mathbb{R}^d}(\mathrm{P}_{\pi^*L - t E}) }{ \operatorname{Vol}_{\mathbb{R}^d}(\mathrm{P}_{\pi^*L}) } \mathrm{d}t \text{.}
$$

\np {\bf Example.} Here we treat the case of a torus invariant point
$x = \{\mathrm{pt} \} \in \PP^2$
and an ample line bundle
$L = \Osh_{\PP^2}(a)$, for $a > 0$.  Our conventions about star subdivisions are as in \S \ref{PP2:star:subdivision}.  By linear equivalence, it suffices to determine
$$
\beta(a \pi^* D_2, E) = \beta(a D_2' , E) \text{.}
$$
The polyhedra $\mathrm{P}_{a D_2}$ is a triangle in the fourth quadrant.  This  triangle has area 
$$
\operatorname{Vol}_{\RR^2}(\mathrm{P}_{a D_2}) = \frac{1}{2}a^2.
$$
Similarly it follows that
$$
\operatorname{Vol}_{\RR^2}( \mathrm{P}_{a D'_2 - t E}) = \frac{1}{2}(a^2 - t^2) 
$$
for $0 \leq t \leq a$.
In particular, it follows that
$$
\beta_x(L) = \int_0^a \frac{ a^2 - t^2 }{ a^2 } \mathrm{d}t = \frac{2}{3}{a} \text{,}
$$
compare with \cite[Section 4]{McKinnon-Roth}.

\np\label{P1:P1:point:example} {\bf Example.}  Next we compute $\beta_x(L)$ for a torus invariant point
$x = \{\mathrm{pt}\} \in \PP^1 \times \PP^1$
and $L$ an ample line bundle.  Our aim is to study the divisors
$$
a \pi^* D_3 + b \pi^* D_4 - t E \text{, }
$$
for 
$a,b \in \ZZ_{>0}$ 
and 
$t \in \RR_{\geq 0}\text{,}$ 
on 
$$
\pi \colon S' = \operatorname{Bl}_{\{\mathrm{pt} \}}(S) \rightarrow S = \PP^1 \times \PP^1
$$
the blowing-up of $S$ at the torus invariant point $\{\mathrm{pt}\} \in S$ which corresponds to the maximal cone which is determined by $v_2$ and $v_3$.  Here, our conventions for the minimal ray generators for $\PP^1 \times \PP^1$ are such that 
$$\text{$v_1 = (-1,0)$, $v_2 = (0,1)$, $v_3 = (1,0)$ and $v_4 = (0,-1)$.}$$

 The goal is to use the polyhedrons 
$$
\mathrm{P}_{a \pi^* D_3 + b \pi^* D_4 - t E} = \phi^{-1}(\Delta \left(a \pi^* D_3 + b \pi^* D_4 - t E) \right)\text{, }
$$ 
or equivalently their interpretation in terms of Okounkov bodies, to determine the nature of the functions
$$
f(t) := \frac{\Vol_{S'}\left( a \pi^* D_3 + b \pi^* D_4 - t E \right)}{\Vol_{S'} \left( a \pi^* D_3 + b \pi^* D_4  \right)} \text{.}
$$
In doing so, we recover similar calculations which were obtained in \cite[Section 4]{McKinnon-Roth}.

Henceforth, we assume that $a \leq b$.  Rewriting, the divisor
$$
a \pi^* D_3 + b \pi^* D_4 \text{, }
$$
in terms of the torus invariant prime divisors on $S'$, we obtain that
\begin{equation}\label{PP1:PP1:pullback:eqn1}
a \pi^* D_3 + b \pi^* D_4 = a D_3' + a E + b D'_4.
\end{equation}

It follows from \eqref{PP1:PP1:pullback:eqn1} that the inequalities which govern the polyhedron
$
\mathrm{P}_{a D_3' + b D_4' + a E}
$
are then given by the following five inequalities
\begin{equation*}
\begin{matrix}
m \cdot v_0'  & = & m_1  & + & m_2 &   \geq & - a \\
m\cdot v_1'  & = & -m_1  & + & 0   &  \geq  & 0 \\
m \cdot v_2'  & = & 0 & + & m_2  & \geq & 0 \\
m \cdot v_3' & = & m_1  & + & 0   & \geq & - a \\
m \cdot v_4'  & = & 0  & - & m_2  & \geq  & - b \text{.}
\end{matrix}
\end{equation*}
In terms of the divisor
$$
a \pi^*D_3 + b \pi^* D_4 - t E = a D'_3 + b D_4' + (a- t) E 
$$
the inequalities that govern its polyhedron 
$$
\mathrm{P}_{a \pi^*D_3 + b \pi^* D_4 - t E} = \mathrm{P}_{a D'_3 + b D_4' +(a - t) E}
$$
are
\begin{equation*}
\begin{matrix}
m \cdot v_0'  & = & m_1  & + & m_2 &   \geq & - a + t \\
m\cdot v_1'  & = & -m_1  & + & 0   &  \geq  & 0 \\
m \cdot v_2'  & = & 0 & + & m_2  & \geq & 0 \\
m \cdot v_3' & = & m_1  & + & 0   & \geq & - a \\
m \cdot v_4'  & = & 0  & - & m_2  & \geq  & - b \text{.}
\end{matrix}
\end{equation*}

There are three cases to consider
\begin{enumerate}
\item[(i)]{
the case that
$0 \leq t \leq a \text{; }
$
}
\item[(ii)]{
the case that
$
a \leq t \leq b \text{; and}
$
}
\item[(iii)]{
the case that
$
t > b \text{.}
$
}
\end{enumerate}
Henceforth, we also set
$
\gamma_{\mathrm{eff}} = a + b \text{.}
$

To begin with, we observe that
$$
\Vol_{\RR^2}(\mathrm{P}_{a \pi^* D_3 + b \pi^* D_4}) = \Vol_{\RR^2}( \mathrm{P}_{a D_3 + b D_4}) =  a b.
$$
Then, by determining the area of the polyhedrons 
$$
\mathrm{P}_{a \pi^*D_3 + b \pi^* D_4 - t E}\text{, }
$$
in each of the above three cases, we obtain that 
$$
f(t) 
=  
\begin{cases}
1 - \frac{t^2}{2 a b} & \text{for $0 \leq t \leq a$}  \\
\frac{1}{ab}\left(ab + \frac{1}{2} a^2 - t a \right)  & \text{for $a \leq t \leq b$} \\
\frac{(a+b - t)^2 }{2 a b }& \text{for $b \leq t \leq a + b$.}
\end{cases}
$$

In particular, we recover the calculations which were obtained in \cite[Section 4]{McKinnon-Roth}.  By integrating $f(t)$ over the interval $[0,a+b]$, we determine the asymptotic volume constant 
\begin{align*}
\beta(aD_3 + b D_4, E)&  := \int_0^{\infty} \frac{\Vol(a \pi^* D_3 + b \pi^* D_4 -t E ) }{ \Vol(a D_3 + b D_4) } \mathrm{d}t \\
& =
\frac{a + b}{2}
\end{align*}
as in \cite[Section 4]{McKinnon-Roth}.

\section{A gcd bound for $\PP^1 \times \PP^1$ blown-up along a torus invariant point}\label{P1:P1:blow:up:gcd}

\np  The aim of this section, is to establish an unconditional bound for the generalized greatest common divisor of pairs of nonzero algebraic numbers.  In what follows, we continue to adopt the notation of Example \ref{P1:P1:point:example}.

\np\label{intersect:prop:eqn1}  Let $p$ and $q$ denote the projections of $S = S(\Sigma) = \PP^1 \times \PP^1$ onto the first and second factors respectively.  Recall, that the anticanonical divisor class is of type $(2,2)$ on $\PP^1 \times \PP^1$.  In particular, by linear equivalence, we may write the anticanonical divisor $-\K_{S(\Sigma)}$ in the form
$$
-\K_{S(\Sigma)} \sim p^*H_1 + p^*H_2 + q^*F_1 + q^*F_2
$$
where the divisors $p^*H_1$, $p^*H_2$, $q^*F_1$, $q^*F_2$ have the property that their pullback to $S' = S(\Sigma') $ intersect properly.

\np\label{intersect:prop:eqn2}  Thus, using linear equivalence, we may write
\begin{equation}\label{pull:back:anti:can:intersect:properly}
- \pi^* \K_{S(\Sigma)} = H'_1 + H_2' + F'_1 + F'_2
\end{equation}
where $H_i'$ denotes the pullback of the type $(1,0)$ divisors $p^*H_i$, for $i = 1,2$, and where $F_i'$ denotes, similarly, the pullback of the type $(0,1)$ divisors $q^*F_i$.

\np  In particular, we may apply Theorem \ref{unconditional:arithmetic:general:theorem:gcd:bound} with respect to the asymptotic volume constants 
$$\beta(-\K_{S(\Sigma')}, H_i')$$ 
and 
$$\beta(-\K_{S(\Sigma')}, F_i')\text{,}$$ 
for $i = 1,2$.  To this end, by linear equivalence, it suffices to determine the nature of the 
$$\beta(-\K_{S(\Sigma')}, \pi^*D_i)\text{,}$$ 
for $i = 1,2,3,4$.

\np  With this end in mind, here we discuss the polyhedron for the ample divisor
$$
- \K_{S'} = - \pi^* \K_S - E = D_1' + D_2' + D_3' + D_4' + E \text{.}
$$
Note that by solving for 
$$- \pi^* \K_S = - \K_{S'} + E \text{,}$$
we obtain
\begin{align*}
- \pi^* \K_S  &
= - \K_{S'} + E 
\\
&
= D_1' + D_2' + D_3' + D_4' + 2 E \\
 &
= D_1' + (D_2' + E) + (D_3' + E) + D_4' \\
& = \pi^* D_1 + \pi^* D_2 + \pi^* D_3 + \pi^* D_4 \text{.}
\end{align*}

\np  We now determine the asymptotic volume constants
$$
\beta(-\K_{S'}, \pi^* D_i) := \int_0^\infty \frac{\Vol_{S'}(-\K_{S'} - t \pi^* D_i ) }{ \Vol_{S'}(-\K_{S'}) } \mathrm{d}t 
$$
for $i = 1,2,3$ and $4$.  We have the known lower bound
$$
\beta(- \K_{S'}, \pi^* D_i) \geq \frac{7}{8} \text{, }
$$
compare with 
\cite[Corollary 1.11]{Ru:Vojta:2016} and \cite[Proposition 12]{Guo:Wang:2017}.

\np  First, we note that the inequalities which determine $\mathrm{P}_{-\K_{S'}}$ are given by
\begin{equation*}
\begin{matrix}
m \cdot v_0'  & = & m_1  & + & m_2 &   \geq & - 1 \\
m\cdot v_1'  & = & -m_1  & + & 0   &  \geq  & -1 \\
m \cdot v_2'  & = &0 & + & m_2  & \geq & -1 \\
m \cdot v_3' & = & m_1  & + & 0   & \geq & - 1 \\
m \cdot v_4'  & = & 0  & - & m_2  & \geq  & \ - 1 \text{.}
\end{matrix}
\end{equation*}
It follows that the polyhedron $\mathrm{P}_{-\K_{S'}}$ has area equal to 
$$
\Vol_{\RR^2}(\mathrm{P}_{-\K_{S'}}) = 1 + 1 + 1 + \frac{1}{2} = \frac{7}{2}
$$
and, furthermore, that
$$
\Vol(-\K_{S'}) = 2 + 2 + 2 + 1 = 7.
$$

\np  Next, we discuss the divisors 
$
- \K_{S'} - t \pi^* D_i \text{, }
$
for $i = 1,2,3$ and $4$.  To begin with, recall that
$$
\pi^* D_3 = D_3' + E \text{, } \pi^* D_2 = D_2' + E \text{, }
\pi^* D_1 = D_1' \text{ and } \pi^*D_4 = D_4' \text{.}
$$

\np  We observe Proposition \ref{beta:gcd:prop} below.
\begin{proposition}\label{beta:gcd:prop}
Let 
$\pi \colon S' = \operatorname{Bl}_{\{\mathrm{pt} \}}(S) \rightarrow S = \PP^1 \times \PP^1$ 
be the blowing-up at a torus invariant point 
$
\{\mathrm{pt}\} \in S
$
with exceptional divisor $E$.  Let $D_i$, for $i = 1,\dots,4$, be the prime torus invariant divisors determined by the primitive ray vectors in $\Sigma$ the fan of $S$.  Then, in this context, it holds true that
$$ 
\beta(-\K_{S'}, \pi^* D_i)
= \int_0^{\infty} \frac{\Vol_{S'}(-\K_{S'} - t \pi^* D_i ) }{ \Vol_{S'}(-\K_{S'}) } \mathrm{d}t = \frac{19}{21} \text{.}
$$
\end{proposition}

\begin{proof}
By linear equivalence and symmetry, 
it suffices to consider the case that $i = 4$ and then determine the areas of the polyhedra $\mathrm{P}_{-\K_{S'} - t \pi^* D_4}$.  With this in mind, we will show that 
$$
\operatorname{Vol}_{\RR^2}(\mathrm{P}_{-\K_{S'} - t \pi^* D_4})  =
\begin{cases}
\frac{7}{2} - 2t & \text{ for $0 \leq t \leq 1$} \\
\frac{1}{2}(2-t)^2 + 2 - t & \text{ for $1 \leq t \leq 2$.}
\end{cases}
$$

To determine the area of $\mathrm{P}_{-\K_{S'} - t \pi^*D_4}$, we need to determine the area 
$\operatorname{Vol}_{\RR^2}(\mathrm{P}_{-\K_{S'} - t \pi^*D_4})\text{,}$
for $0 \leq t \leq 2$, of the region that is bounded by the five inequalities
\begin{equation*}
\begin{matrix}
m_1  & + & m_2 &   \geq & - 1 \\
 -m_1  & + & 0   &  \geq  & -1 \\
0 & + & m_2  & \geq & -1 \\
m_1  & + & 0   & \geq & - 1 \\
0  & - & m_2  & \geq  & - 1 + t \text{.}
\end{matrix}
\end{equation*}

For $1 \leq t \leq 2$, we study the fourth and third quadrants.  The areas of the regions that are bounded by these inequalities may be described in the following way
\begin{itemize}
\item{
the region in the fourth quadrant that is bounded by these inequalities has area equal to 
$
2-t;
$
}
\item{
the region in the third quadrant that is bounded by these inequalities has area equal to
$
\frac{1}{2}(2-t)^2.
$
Combining, it follows that
$$
\operatorname{Vol}_{\RR^2}(\mathrm{P}_{-\K_{S'} - t \pi^* D_4}) = \frac{1}{2}(2-t)^2 + 2 - t
\text{, }
$$
for $1 \leq t \leq 2$.
}
\end{itemize}
For $0 \leq t \leq 1$, we note
\begin{itemize}
\item{
the region in the first quadrant that is bounded by these inequalities has area equal to
$
1-t;
$
}
\item{
the region in the second quadrant that is bounded by these inequalities has area equal to
$
1-t;
$
}
\item{
the region in the third quadrant that is bounded by these inequalities has  area equal to
$
\frac{1}{2};
$
}
\item{
the region in the fourth quadrant that is bounded by these inequalities has area equal to 
$1$.
}
\end{itemize}
In sum, it follows that
$$
\operatorname{Vol}_{\RR^2}(\mathrm{P}_{-\K_{S'} - t \pi^* D_4}) = \frac{7}{2} - 2t \text{,}
$$
for $0 \leq t \leq 1$.

We now calculate the expected order of vanishing of $-\K_{S'}$ along $\pi^* D_4$.   
To begin with, 
set
$$
f(t) = \frac{\Vol_{S'}(-\K_{S'} - t \pi^* D_4)}{\Vol_{S'}(-\K_{S'})}.
$$
Then  
$$
\Vol_{S'}(-\K_{S'}) = 7
$$
and 
$$
f(t) = \frac{2}{7} \operatorname{Vol}_{\RR^2}(\mathrm{P}_{-\K_{S'} - t \pi^* D_4}).
$$
Thus, it then follows that
$$
\int_0^2 f(t) \mathrm{d}t = \frac{2}{7}\left( \frac{19}{6} \right) = \frac{19}{21}\text{,}
$$
as is desired by Proposition \ref{beta:gcd:prop}.
\end{proof}

\np  Next, we combine Proposition \ref{beta:gcd:prop} and Theorem \ref{unconditional:arithmetic:general:theorem:gcd:bound} to obtain a gcd type bound for the case of $\PP^1 \times \PP^1$ blown-up along a torus invariant point.  The idea is to apply Theorem \ref{unconditional:arithmetic:general:theorem:gcd:bound}.  In particular, we first note that a consequence of Proposition \ref{beta:gcd:prop} is the fact that
$$
 \beta(-\K_{S(\Sigma')},\pi^*D_i) = \frac{19}{21} \text{, }
$$
for $i = 1,\dots, 4$.  Next, as mentioned in Subsections \ref{intersect:prop:eqn1} and \ref{intersect:prop:eqn2}, if we use linear equivalence to write the anticanonical divisor in the form
$$
-\K_{S(\Sigma)} = p^*H_1 + p^* H_2 + q^* F_1 + q^* F_2 \text{,}
$$
it follows that the divisors 
$H_1', H_2', F_1', F_2'$ intersect properly.

\np
Hencefourth, we put
$$
\gamma := \frac{21}{19}
$$
\text{ 
and 
}
$$
\delta := \gamma - 1 = \frac{21}{19} - 1 
= \frac{2}{19}\text{.}
$$
In our present context, Theorem \ref{unconditional:arithmetic:general:theorem:gcd:bound} takes the following form.

\begin{theorem}\label{gcd:bound:intro:thm'}
Let $\kk$ be a number field, $M_\kk$ its set of places and $S \subseteq M_\kk$ a finite subset.  Let 
$S(\Sigma) := \PP^1_\kk \times \PP^1_\kk$ 
and
$ 
\pi \colon S(\Sigma') \rightarrow S(\Sigma)
$ 
the blowing-up of $S(\Sigma)$ at a torus fixed point 
$\{\mathrm{pt} \} \in S(\Sigma)$ 
with exceptional divisor $E$.  Let $\epsilon > 0$ and put $\delta =  \frac{2}{19} $.  Then there exists a proper Zariski closed subset 
$Z' \subsetneq S(\Sigma')$ 
with the property that
\begin{equation}\label{Silverman:gcd:inequality}
h_E(x') \leq  \frac{1}{1 + \delta + \epsilon} \left( ( \delta + \epsilon)  h_{-\pi^* \K_{S(\Sigma)}}(x') +  \left( \sum_{v \in M_\kk \setminus S} \lambda_{\pi^* \K_{S(\Sigma)}}(x') \right) \right) + \mathrm{O}(1)
\end{equation}
for all $\kk$-rational points  $x' \in S(\Sigma')(\kk) \setminus Z'(\kk)$.
\end{theorem}

\begin{proof}
Because of our discussions that precede Theorem  \ref{gcd:bound:intro:thm'},  the hypothesis of Theorem \ref{unconditional:arithmetic:general:theorem:gcd:bound} is satisfied, writing $-\pi^*\K_{S(\Sigma)}$ as in \eqref{pull:back:anti:can:intersect:properly}, and moreover, in  \eqref{gcd:eqn12'}, we may put $\delta = 
\frac{2}{19}
$.  In doing so, we obtain the desired inequality \eqref{Silverman:gcd:inequality}.  
\end{proof}

\np{\bf Acknowledgements.}  
This work was conducted while I was a postdoctoral fellow at Michigan State University.  It is my pleasure to thank many colleagues for their interest and discussions on related topics.  Finally, I thank an anonymous referee for carefully reading this article and for offering several helpful comments and suggestions. 

\providecommand{\bysame}{\leavevmode\hbox to3em{\hrulefill}\thinspace}
\providecommand{\MR}{\relax\ifhmode\unskip\space\fi MR }
\providecommand{\MRhref}[2]{%
  \href{http://www.ams.org/mathscinet-getitem?mr=#1}{#2}
}
\providecommand{\href}[2]{#2}

\end{document}